\documentclass{article}
\usepackage{latexsym,amsthm,amssymb,amscd,amsmath}
\newcounter{alphthm}
\newtheorem{thm}{Theorem}

\newtheorem{prop}{Proposition}
\newtheorem{cor}{Corollary}
\newtheorem{lem}{Lemma}

\newcommand{\be}{\begin{equation}}
\newcommand{\ee}{\end{equation}}
\newcommand{\ben}{\begin{enumerate}}
\newcommand{\een}{\end{enumerate}}

\def\beq{\begin{equation}}
\def\eeq{\end{equation}}

\begin{document}
\title{$(1,1)$-Tensor sphere bundle of Cheeger-Gromoll type}
\author{E. Peyghan, L. Nourmohammadi Far, A. Tayebi}

\maketitle
\begin{abstract}
We construct a metrical framed $f(3,-1)$-structure on the (1, 1)-tensor
bundle of a Riemannian manifold equipped with a Cheeger-Gromoll type
metric and by restricting this structure to the (1, 1)-tensor sphere bundle,
we obtain an almost metrical paracontact structure on the (1, 1)-tensor
sphere bundle. Moreover, we show that the (1, 1)-tensor sphere bundles
endowed with the induced metric are never space forms.
\end{abstract}

\section{Introduction}
Maybe the best known Riemannian metric on the tangent bundle is that introduced
by Sasaki in 1958 (see \cite{S}), but in most cases the study of some geometric
properties of the tangent bundle endowed with this metric led to the flatness
of the base manifold. In the next years, the authors were interested in finding
other lifted structures on the tangent bundles, cotangent and tangent sphere bundles with quite interesting properties
(see \cite{AS1}, \cite{CG}-\cite{DO2}, \cite{Se}).

The tangent sphere bundle $T_rM$ consisting of spheres of constant radius
r seen as hypersurfaces of the tangent bundle $TM$ have important applications in
geometry. In the last years some interesting results were obtained by endowing
the tangent sphere bundles with Riemannian metrics induced by the natural
lifted metrics from TM, which are not Sasakian (see \cite{AK}, \cite{DO}, \cite{M}).

Tensor bundles $T^p_qM$ of type $(p, q)$ over a differentiable manifold $M$ are prime
examples of fiber bundles, which are studied by mathematicians such as Ledger,
Yano, Cengiz and Salimov \cite{CS}, \cite{LY}, \cite{SC}. The tangent bundle $TM$ and cotangent
bundle $T^*M$ are the special cases of $T^p_qM$.

In \cite{SG}, Salimov and Gezer introduced the Sasaki metric $^Sg$ on the (1, 1)-tensor
bundle $T^1_1M$ of a Riemannian manifold $M$ and studied some geometric properties
of this metric. By the similar method used in the tangent bundle, the same
authors defined in \cite{PTN}  the Cheeger-Gromoll type metric $^{CG}g$ on $T^1_1M$ which is
an extension of Sasaki metric. Then they studied some relations between the
geometric properties of the base manifold $(M, g)$ and $(T^1_1M, {}^{CG}g)$.
In the present paper, we consider Cheeger-Gromoll type metric $^{CG}g$ on $T^1_1M$
and using it we introduce a metrical framed $f(3,-1)$-structure on $T^1_1M$. Then,
by restricting this structure to the (1, 1)-tensor sphere bundle of constant radius $r$, $T^1_{1r}M$, we obtain
a metrical almost paracontact structure on $T^1_{1r}M$. Finally, we show that the
(1, 1)-tensor sphere bundles endowed with the induced metric are never space
forms.
\section{Preliminaries}
Let $M$ be a smooth $n$-dimensional manifold. We define the bundle of $(1, 1)$-tenors on $M$ as $T^1_1M=\coprod_{p\in M}T^1_1(p)$, where $\coprod$
denotes the disjoint union, and we call it $(1, 1)$-tensor bundle. We define also the projection $\pi:T^1_1M\rightarrow M$
 to $p$. If $(x^i)$ are any local coordinates on $U\subset M$, and $p\in U$, the coordinate vectors $\{\partial_i\}$, where $\partial_i:=\frac{\partial}{\partial x^i}$, 
 form a basis for $T_pM$ whose dual basis is $dx^i$. Any tensor $t\in T^1_1M$ can be
expressed in terms of this basis as $t=t^i_j\partial_i\otimes dx^j$.

For any coordinate chart $(U, (x^i))$ on $M$, correspondence $t\in T^1_1(x)\rightarrow (x, (t^i_j))\in U\times R^{n^2}$
determines local trivializations $\phi:\pi^{-1}(U)\subset T^1_1M\rightarrow U\times R^{n^2}$ , that is, $T^1_1M$ is a vector bundle on $M$.
Therefore each local coordinate neighborhood  $\{(U, x^j)\}_{j=1}^n$ in $M$ induces on $T^1_1M$ a local coordinate
neighborhood $\{\pi^{-1}(U);\ \  x^j,\ x^{\bar j}=t^i_j\}_{j=1}^n$, $\bar{j}=n+j$, that is, $T^1_1M$ is a smooth manifold of dimension $n+n^2$.

We denote by $F(M)$ and $\Im^1_1(M)$, the ring of real-valued $C^\infty$ functions and the space of all $C^\infty$ tensor fields of type $(1, 1)$ on $M$. If
$\alpha\in\Im^1_1(M)$, then by contraction it is regarded as a
function on $T^1_1M$, which we denote it by $\imath\alpha$. If
$\alpha$ has the local expression
$\alpha=\alpha_i^j\frac{\partial}{\partial x^j}\otimes dx^i$ in a
coordinate neighborhood $U(x^j)\subset M$, then
$\imath(\alpha)=\alpha(t)$ has the local expression
$\imath\alpha=\alpha_i^jt^i_j$ with respect to the coordinates
$(x^j, x^{\bar j})$ in $\pi^{-1}(U)$.

Suppose that $A\in\Im^1_1(M)$. Then the vertical lift $^VA\in\Im^1_0(T^1_1M)$ of $A$ has the following local expression with respect to the
coordinates $(x^j, x^{\bar j})$ in $T^1_1M$
\begin{equation}\label{2.1}
^VA={}^VA^{\bar j}\partial_{\bar j},
\end{equation}
where $^VA^{\bar j}=A^i_j$ and $\partial_{\bar j}:=\frac{\partial}{\partial x^{\bar j}}=\frac{\partial}{\partial t^i_j}$. Moreover,
if $V\in\Im^1_0(M)$, then the complete lift $^CV$ and the horizontal lift $^HV\in\Im^1_0(T^1_1M)$ of $V$ to
$T^1_1M$ have the following local expression with respect to the
coordinates $(x^j, x^{\bar j})$ in $T^1_1M$ (see \cite{CS} and \cite{LY})
\begin{align}
^CV&=V^j\partial_j+(t^m_j(\partial_mV^i)-t^i_m(\partial_jV^m))\partial_{\bar j},\label{2.2}\\
^HV&=V^j\partial_j+V^s(\Gamma^m_{sj}t^i_m-\Gamma^i_{sm}t^m_j)\partial_{\bar j},\label{2.3}
\end{align}
where $\Gamma^k_{ij}$ are the local components of $\nabla$ on $M$.

Let $U(x^h)$ be a local chart of $M$. By using (\ref{2.1}) and
(\ref{2.3}) we obtain
\begin{eqnarray}
e_j:\!\!\!\!&=&\!\!\!\!{}^H\partial_j={}^H(\delta^h_j\partial_h)=\delta^h_j\partial_h+(\Gamma^s_{jh}t^k_s-\Gamma^k_{js}t^s_h)\partial_{\bar
h},\\
e_{\bar{j}}:\!\!\!\!&=&\!\!\!\!{}^V(\partial_i\otimes
dx^j)={}^V(\delta^k_i\delta^j_h\partial_k\otimes
dx^h)=\delta^k_i\delta^j_h\partial_{\bar h},
\end{eqnarray}
where $\delta^h_j$ is the Kronecker's symbol and
$\bar{j}=n+1,\ldots,n+n^2$. These $n+n^2$ vector fields are
linearly independent and generate, respectively, the horizontal
distribution of $\nabla$ and vertical distribution of $T^1_1M$.
Indeed, we have $^HX=X^je_j$ and $^VA=A^i_je_{\bar{j}}$ (see
\cite{SG}).
 The set $\{e_\beta\}=\{e_{j}, e_{\bar{j}}\}$ is called the frame adapted to the affine connection $\nabla$ on $\pi^{-1}(U)\subset T^1_1M$.
 \begin{lem}\label{esileila}
Let $\alpha_1$, $\alpha_2$, $\alpha_3$, $\alpha_4$ be smooth functions on $T^1_1M$ such that
\begin{equation}\label{s1}
\alpha_1g_{ti}g^{lj}\delta^m_r\delta^v_n+\alpha_2g_{ni}g^{mj}\delta^l_r\delta^v_t+
\alpha_3\overline{t}^m_n\overline{t}^j_i\delta^l_r\delta^v_t
+\alpha_4\overline{t}^l_t\overline{t}^j_i\delta^m_r\delta^v_n=0.
\end{equation}
Then $\alpha_1=\alpha_2=\alpha_3=\alpha_4=0$.
\end{lem}
\begin{proof}
Contacting (\ref{s1}) with $\overline{t}^r_v$, then
differentiating the obtained expression three times, it follows
that, $\alpha_3=-\alpha_4$. Also differentiating the remaining
expression two times, we have
\[
\alpha_1g_{ti}g^{lj}\overline{t}^m_n-\alpha_2g_{ni}g^{mj}\overline{t}^l_t=0,
\]
Contacting the above equation with $t^j_i$, yield
$\alpha_1=-\alpha_2$.
Multiplying (\ref{s1}) by $g_{jh}g^{ik}$,
$\delta_m^h\delta^n_k$, we obtain
$\alpha_3=\alpha_4=0$. Finally contacting (\ref{s1}) with $t^j_i$, $t^m_n$, we conclude that,
$\alpha_1=\alpha_2=0$.
\end{proof}
 \section{ Cheeger-Gromoll type metric on $T_1^1M$}
For each $p\in M$ the extension of the scalar product $g$, denoted by
$G$, is defined on the tensor space $\pi^{-1}(p)=T^1_1(p)$ by
\[
G(A, B)=g_{it}g^{jl}A^i_jB^t_l,\ \ \ A, B\in\Im^1_1(p),
\]
where $g_{ij}$ and $g^{ij}$ are the local covariant and
contravariant tensors associated to the metric $g$ on $M$.

Now, we consider on $T^1_1M$ a Riemannian metric $^{CG}g$ of Cheeger-Gromoll type, as follows:
\begin{equation}\label{3.7}
\left\{
\begin{array}{cc}
^{CG}g(^VA,^VB)={}^V\Big(a G(A,B)+b G(t,A)G(t,B)\Big),&\\
\hspace{-3cm} ^{CG}g(^HX,^HY)={}^V(g(X,Y)),&\\
\hspace{-4.5cm} ^{CG}g(^VA,^HY)=0,&
\end{array}
\right.
\end{equation}
for each $X, Y\in\Im^1_0(M)$ and $A, B\in\Im^1_1(M)$, where $a$ and
$b$ are smooth functions of
$\tau=||t||^2=t^i_jt^t_lg_{it}(x)g^{jl}(x)$ on $T^1_1M$ that
satisfy the conditions $a>0$ and $a+b\tau>0$.

The symmetric matrix of type $2n\times 2n$
\begin{equation}
\left(
\begin{array}{cc}
g_{jl}&0\\
0&ag^{jl}g_{it}+b\bar{t}^j_i\bar{t}^l_t
\end{array}
\right),
\end{equation}
associated to the metric $^{CG}g$ in the adapted frame $\{e_\beta\}$,
has the inverse
\begin{equation}
\left(
\begin{array}{cc}
g^{jl}&0\\
0&\frac{1}{a}g_{jl}g^{it}-\frac{b}{a(a+b\tau)}t^i_jt^t_l
\end{array}
\right),
\end{equation}
where $\bar{t}^j_i=g^{jh}g_{ik}t^k_h$. In the special case, if $a=1$ and $b=0$, we have the Sasaki metric $^Sg$ (see \cite{SG}).

Let $\varphi=\varphi^i_j\frac{\partial}{\partial x^i}\otimes dx^j$
be a tensor field on $M$. Then
$\gamma\varphi=(t^m_j\varphi^i_m)\frac{\partial}{\partial x^{\bar
j}}$ and
$\widetilde{\gamma}\varphi=(t^i_m\varphi^m_j)\frac{\partial}{\partial
x^{\bar j}}$ are vector fields on $T^1_1M$. The bracket
operation of vertical and horizontal vector fields is given by the
formulas
\begin{eqnarray}
&&[^VA, ^VB]=0,\ \ [^HX, ^VA]=^V(\nabla_XA),\label{bracket0}\\
&&[^HX, ^HY]=^H[X, Y]+(\widetilde{\gamma}-\gamma)R(X, Y),\label{bracket}
\end{eqnarray}
where $R$ denotes the curvature tensor field of the connection
$\nabla$ and $\widetilde{\gamma}-\gamma:
\varphi\rightarrow\Im^1_0(T^1_1M)$ is the operator defined by
\[
(\widetilde{\gamma}-\gamma)\varphi=\left(
\begin{array}{c}
0\\
t^i_m\varphi^m_j-t^m_j\varphi^i_m
\end{array}
\right),\ \ \ \forall\varphi\in\Im^1_1(M).
\]
\begin{prop}
The Levi-Civita connection $^{CG}\nabla$ associated to the Riemannian
metric $^{CG}g$ on the $(1, 1)$-tensor bundle $T^1_1M$ has the form
\begin{eqnarray*}
{}^{CG}{\nabla}_{e_l}^{e_j}\!\!\!\!&=&\!\!\!\!\Gamma^r_{lj}e_r+\frac{1}{2}(R_{ljr}^{\
\ \ s}t^v_s-R_{ljs}^{\ \ \ v}t^s_r)e_{\overline{r}},\nonumber\\
{}^{CG}{\nabla}_{e_{\overline{l}}}^{e_j}\!\!\!\!&=&\!\!\!\!\frac{a}{2}(g_{ta}R_{\
\ j }^{sl\ r}t^a_s-g^{lb}R_{tsj}^{\ \ \ r}t^s_b)e_r
,\nonumber\\
{}^{CG}{\nabla}_{e_l}^{e_{\overline{j}}}\!\!\!\!&=&\!\!\!\!\frac{a}{2}(g_{ia}R_{\
\ l }^{sj\ r}t^a_s-g^{jb}R_{isl}^{\ \ \ r}t^s_b)e_r
+(\Gamma^v_{li}\delta^j_r-\Gamma^j_{lr}\delta^v_i)e_{\overline{r}},\nonumber\\
{}^{CG}{\nabla}_{e_{\overline{l}}}^{e_{\overline{j}}}
\!\!\!\!&=&\!\!\!\!(L(\overline{t}^l_t\delta^j_r\delta^v_i+\overline{t}^j_i\delta^l_r\delta^v_t)+M
g^{lj}g_{ti}t^v_r+N\overline{t}^l_t\overline{t}^j_it^v_r)e_{\overline{r}},
\end{eqnarray*}
where $R_{ljr}^{\ \ \ s}$ are the components of the curvature tensor field of the Levi-Civita
connection on the base manifold $(M, g)$ and $L:=\frac{a'}{a}$, $M:=\frac{-a'+2b}{a+b\tau}$,
$N:=\frac{b'a-2a'b}{a(a+b\tau)}$.
\end{prop}
In the following sections we consider the subset $T_{1r}^1M$ of
$T_{1}^1M$ consisting of sphere of constant radius $r$.
Now, we consider the $(1, 1)$- tensor field P on $T^1_1M$ as follows
\[
\left\{
\begin{array}{ccc}
P^HX=c_1{}^V(X\otimes \widetilde{E})+d_1g(X,E){}^V(E\otimes \widetilde{E}),\\
\hspace{-1cm}P^V(X\otimes \widetilde{E})=c_2{}^HX+d_2g(X,E){}^HE,\\
\hspace{-4.5cm}P(^VA)={}^VA,\\
\end{array}
\right.
\]
where $c_1$, $c_2$, $d_1$, $d_2$ are smooth functions of the energy density $t$ and $\widetilde{E}=g\circ E\in\Im^0_1(M)$. Using
the adapted frame $\{e_i, E_je_{\bar j}, e_{\bar j}\}$ to $T^1_1M$, $P$ has the following locally expression
\begin{equation}\label{3.12}
\left\{
\begin{array}{ccc}
P(e_i)=c_1E_je_{\bar{j}}+d_1E_iE^vE_re_{\bar{r}},\\
\hspace{-0.5cm}P(E_je_{\bar{j} })=c_2e_i+d_2E_iE^re_r,\\
\hspace{-3cm}P(e_{\bar{r}})=e_{\bar{r}},\\
\end{array}
\right.
\end{equation}
where $E_k=g_{rk}E^r$. In \cite{PTN}, the following theorem proved.
\begin{thm}\label{Th}
The natural tensor field $P$ of type (1, 1) on $T^1_1M$, defined
by the relations (\ref{3.12}), is an almost product structure
on $T^1_1M$, if and only if its coefficients are related by
\begin{equation}\label{3.13}
c_1c_2=1,\ \ \ \ (c_1+d_1||E||^2)(c_2+d_2||E||^2)=1.
\end{equation}
\end{thm}
\begin{thm}\label{TH00}
$({}^{CG}g, P)$ is an Riemannian almost product structure on $T^1_1M$ if and only if
\begin{equation}\label{3.14}
c_1=\frac{1}{\sqrt{a}||E||},\ c_2=||E||\sqrt{a},\ d_1=\frac{-2}{\sqrt{a}||E||^3},\ d_2=\frac{-2\sqrt{a}}{|E||},
\end{equation}
and (\ref{3.13}) hold good.
\end{thm}
Now, we consider vector fields
fields on $T^1_1M$:
\begin{equation}\label{3.15}
\xi_{1}:= \alpha{}^HE,\ \ \ \
\xi_{2}:=\beta{}^V(E\otimes \widetilde{E}),\ \ \ \
\xi_{3}:=\kappa{}^VA,
\end{equation}
and 1-forms
\begin{equation}\label{3.16}
\eta^1=\gamma E_vdx^{v},\ \ \
\eta^2=\lambda E_vE^r\delta t^v_r,\ \ \ \eta^3=\rho\bar{t}^r_v\delta t^v_r,
\end{equation}
on $T^1_1M$, where $\alpha$, $\beta$, $\kappa$, $\gamma$, $\lambda$, $\rho$ are smooth functions of
the energy density on $T^1_1M$ and $\delta t^v_r$ is a dual of $e_{\bar r}$. Using (\ref{3.12}) and
(\ref{3.15}), we get
\begin{equation}\label{3.17}
P(\xi_{1})=\frac{\alpha}{\beta}( c_1+d_1||E||^2)\xi_2,\ \ \
P(\xi_2)=\frac{\beta}{\alpha}(c_2+d_2||E||^2)\xi_1, \ \ \
P(\xi_{3})=\xi_{3},
\end{equation}
and
\begin{equation}\label{3.18}
\eta^1(\xi_1)=\alpha\gamma ||E||^2,\ \ \
\eta^2(\xi_2)=\beta\lambda ||E||^4,\ \ \ \eta^3(\xi_3)=\kappa\rho
\tau,\ \ \eta^a(\xi_b)=0,
\end{equation}
where $a, b=1, 2, 3$ with condition $a\neq b$. We have also by (\ref{3.12}) and (\ref{3.16})
\begin{equation}\label{3.19}
\eta^1\circ P=\frac{\gamma}{\lambda
||E||^2}(c_2+d_2||E||^2)\eta^2,\ \ \ \eta^2\circ P=\frac{\lambda
||E||^2 }{\gamma}(c_1+d_1||E||^2)\eta^1,\ \ \ \eta^3\circ
P=\eta^3.
\end{equation}
Now, we define a tensor field $p$ of type (1,1) on $T^1_1M$ by
\begin{equation}\label{3.20}
p(X)=P(X)-\eta^{1}(X)\xi_{2}-\eta^{2}(X)\xi_{1}-\eta^{3}(X)\xi_{3}.
\end{equation}
This can be written in a more compact from as
$p=P-\eta^{1}\otimes \xi_2-\eta^{2}\otimes \xi_1-\eta^{3}\otimes
\xi_3$. From (\ref{3.20}) the following local expression of
$p$ yields
\begin{equation}\label{3.21}
\left\{
\begin{array}{ccc}
\hspace{-.7cm}p(e_i)=\Big(c_1\delta^v_i+(d_1-\beta\gamma)E_{i}E^v\Big)E_re_{\bar
r},\\
p(E_je_{\bar
j})=\Big(c_2\delta^r_i+(d_2-\alpha\lambda ||E||^2)E_{i}E^r\Big)e_r,\\
\hspace{-2.5cm}p(e_{\bar
j})=\Big(\delta^j_r\delta^v_i-\kappa\rho
\bar{t}^j_it^v_r\Big)e_{\bar{r}}.\\
\end{array}
\right.
\end{equation}
\begin{lem}\label{naderi4}
We have
\begin{equation}\label{3.22}
\left\{
\begin{array}{ccc}
\hspace{-.1cm}p(\xi_1)=\frac{\alpha}{\beta}\Big(c_1+(d_1-\beta\gamma)||E||^2\Big)\xi_2,\\
p(\xi_2)=\frac{\beta}{\alpha}\Big(c_2+(d_2-\alpha\lambda||E||^2)||E||^2
\Big)\xi_1,\\
\hspace{-2cm}p(\xi_3)=(1-\kappa\rho\tau)\xi_3,\\
\end{array}
\right.
\end{equation}
\begin{equation}\label{3.23}
\left\{
\begin{array}{ccc}
\eta^1\circ p=\frac{\gamma}{\lambda||E||^2}\Big(c_2+(d_2-\alpha\lambda||E||^2)||E||^2\Big)\eta^2,\\
\hspace{-.7cm}\eta^2\circ p=\frac{\lambda||E||^2}{\gamma}\Big(c_1+(d_1-\beta\gamma)||E||^2\Big)\eta^1,\\
\hspace{-3cm}\eta^3\circ p=\Big(1-\kappa\rho\tau\Big)\eta^3,\\
\end{array}
\right.
\end{equation}
\begin{align}\label{3.24}
p^2&=I-\Big(\frac{\beta}{\alpha}(c_2+d_2||E||^2)+\frac{\lambda||E||^2}{\gamma}(c_1+d_1||E||^2)-\beta\lambda||E||^4\Big)\eta^1\otimes\xi_1\nonumber\\
&\ \ \ -\Big(\frac{\alpha}{\beta}(c_1+d_1||E||^2)+\frac{\gamma}{\lambda||E||^2}(c_2+d_2||E||^2)-\alpha\gamma||E||^2\Big)\eta^2\otimes\xi_2,\nonumber\\
&\ \ \ +(\kappa\rho\tau-2)\eta^3\otimes\xi_3.
\end{align}
\end{lem}
\begin{proof}
We only prove (\ref{3.24}). Using (\ref{3.17}), (\ref{3.18}) and (\ref{3.19}) we have
\begin{align*}
p^2(X)&=p(p(X))=P[P(X)-\eta^1(X)\xi_2-\eta^2(X)\xi_1-\eta^3(X)\xi_3]\\
&\ \ -\eta^1[P(X)-\eta^2(X)\xi_1]\xi_2-\eta^2[P(X)-\eta^1(X)\xi_2]\xi_1\\
&\ \ -\eta^3[P(X)-\eta^3(X)\xi_3]\xi_1=X-\frac{\beta}{\alpha}(c_2+d_2||E||^2)\eta^1(X)\xi_1\\
&\ \ -\frac{\alpha}{\beta}(c_1+d_1||E||^2)\eta^2(X)\xi_2-\frac{\gamma}{\lambda||E||^2}(c_2+d_2||E||^2)\eta^2(X)\xi_2\\
&\ \ +||E||^2\alpha\gamma\eta^2(X)\xi_2-2\eta^3(X)\xi_3-\frac{\lambda||E||^2}{\gamma}(c_1+d_1|E||^2)\eta^1(X)\xi_1\\
&\ \ +||E||^4\beta\lambda\eta^1(X)\xi_1+\kappa\rho\tau\eta^3(X)\xi_3.
\end{align*}
The above equation gives us (\ref{3.24}).
\end{proof}
\begin{lem}\label{binam}
Let $P$ satisfy Theorem \ref{Th}. If
\begin{equation}\label{3.25}
\alpha\gamma||E||^2=1,\ \ \ \beta\lambda||E||^4=1,\ \
\kappa\rho\tau=1,\ \ \ \  \lambda=\frac{\gamma}{||E||^2}(c_2+
d_2||E||^2),
\end{equation}
then $p^3-p=0$ and $p$ has the rank $n+n^2-3$ (or corank 3).
\end{lem}
\begin{proof}
If (\ref{3.25}) holds, then from the above lemma we obtain
\begin{equation}\label{3.26}
p^2=I-\eta^1\otimes\xi_1-\eta^2\otimes\xi_2-\eta^3\otimes\xi_3,\ \
\ p(\xi_k)=0,\ \ \ \eta^k(\xi_l)=\delta^k_l,\ \ \ \eta^k\circ
p=0,
\end{equation}
where $k, l=1,2,3$. Therefore we have $p^3=p$. In order to prove the second part of the lemma, it is sufficient to show that $\ker p=span\{\xi_1, \xi_2, \xi_3\}$.
 From the second relation in (\ref{3.26}) we notice that $span\{\xi_1, \xi_2, \xi_3\}\subset\ker p$. Now we let
$X=X^re_r+X^vE_re_{\bar r}+X^{\bar{r}}e_{\bar{r}}\in\ker p$. Then
$p(X)=0$ implies that
\[
P(X)-\eta^{1}(X)\xi_{2}-\eta^{2}(X)\xi_{1}-\eta^3\otimes\xi_3=0.
\]
Thus
\[
P^2(X)=\eta^1(X)P(\xi_2)+\eta^2(X)P(\xi_1)+\eta^3(X)P(\xi_3).
\]
Since $P^2=I$, then by using (\ref{3.17}) we get
\[
X=\frac{\beta}{\alpha}(c_2+d_2||E||^2)\eta^1(X)\xi_1+\frac{\alpha}{\beta}(c_1+d_1||E||^2)\eta^2(X)\xi_2+\eta^3(X)\xi_3,
\]
that is $X\in span\{\xi_1, \xi_2,\xi_3\}$, i.e., $\ker p\subseteq
span\{\xi_1,\xi_2,\xi_3\}$.
\end{proof}
\begin{thm}\label{t1}
Let  $P$ be the almost product structure characterized in Theorem
\ref{Th} and $\xi_k$, $\eta^k$, $k=1, 2, 3$ and $p$ be defined by
(\ref{3.15}), (\ref{3.16}) and (\ref{3.20}),
respectively. Then the triple $(p, (\xi_k),(\eta^k))$ provides a
framed $f(3, -1)$- structure if and only if (\ref{3.25}) holds.
\end{thm}
\begin{proof}
Let $(p, (\xi_k),(\eta^k))$ be a framed $f(3,-1)$- structure on
$T^1_1M$. Then by the definition of a framed $f(3,-1)$- structure,
we have $\eta^k(\xi_l)=\delta^k_l$, where $k,l=1,2,3$. Thus
 (\ref{3.18}) gives us
\begin{equation}\label{hi}
\alpha\gamma||E||^2=\beta\lambda||E||^4=\kappa\rho\tau=1.
\end{equation}
We have also $p(\xi_3)=0$. The above equation and the second relation in (\ref{3.22}) yield
$\lambda=\frac{\gamma}{||E||^2}(c_2+ d_2||E||^2)$. By using  lemmas \ref{naderi4} and \ref{binam}, the converse of the theorem is proved.
\end{proof}
\begin{lem}\label{Dr}
Let $({}^{CG}g, P)$ satisfy Theorem \ref{TH00}. Then the Riemannian metric ${}^{CG}g$ satisfies
\begin{eqnarray*}
^{CG}g(pX, pY)\!\!\!\!&=&\!\!\!\!{}^{CG}g(X,
Y)-a\beta(\frac{2(c_1+d_1||E||^2)}{\gamma}-\beta||E||^2
)||E||^2\eta^1(X)\eta^1(Y)\\
\!\!\!\!&&\!\!\!\!-\alpha(\frac{2(c_2+d_2||E||^2)}{\lambda||E||^2}-\alpha||E||^2)\eta^2(X)\eta^2(Y)\\
\!\!\!\!&&\!\!\!\!-\kappa(a+b\tau)(\frac{2}{\rho}-\kappa\tau)\eta^3(X)\eta^3(Y),
\end{eqnarray*}
for each $X,Y\in\Im^1_0(T^1_1M)$.
\end{lem}
\begin{proof}
Obviously, we have ${}^{CG}g(\xi_1, \xi_2)=0$. Using (\ref{3.15}), we deduce
\[
{}^{CG}g(\xi_1, \xi_1)=\alpha^2||E||^2,\ {}^{CG}g(\xi_2,
\xi_2)=a\beta^2||E||^4,\ {}^{CG}g(\xi_3,
\xi_3)=\kappa^2(a+b\tau)\tau.
\]
We have also
\[
{}^{CG}g(X,\xi_1)=\frac{\alpha}{\gamma}\eta^1(X),\
{}^{CG}g(X,\xi_2)=\frac{a\beta}{\lambda}\eta^2(X), \
{}^{CG}g(X,\xi_3)=\frac{\kappa}{\rho}(a+b\tau)\eta^3(X).
\]
Using (\ref{3.19}) and the above equations we deduce
\begin{eqnarray*}
{}^{CG}g(pX, pY)\!\!\!\!&=&\!\!\!\!{}^{CG}g(PX, PY)-\frac{2a\beta}{\gamma}(c_1+d_1||E||^2)||E||^2\eta^1(X)\eta^1(Y)\\
\!\!\!\!&&\!\!\!\!+\alpha^2||E||^2\eta^2(X)\eta^2(Y)+a\beta^2||E||^4\eta^1(X)\eta^1(Y)\\
\!\!\!\!&&\!\!\!\!-\frac{2\alpha}{\lambda||E||^2}(c_2+d_2||E||^2)\eta^2(X)\eta^2(Y)\\
\!\!\!\!&&\!\!\!\!-\kappa(a+b\tau)(\frac{2}{\rho}-\kappa\tau)\eta^3(X)\eta^3(Y).
\end{eqnarray*}
But ${}^{CG}g(PX, PY)={}^{CG}g(X, Y)$, since $({}^{CG}g, P)$ is a Riemannian almost product structure. Thus the lemma is proved.
\end{proof}
\begin{thm}
If $({}^{CG}g, P)$ is the Riemannian almost product structure characterized in Theorem \ref{TH00}, and $\xi_k$, $\eta^k$, $k=1, 2, 3$, $p$ are defined by (\ref{3.15}), (\ref{3.16}) and (\ref{3.20}), respectively, then $({}^{CG}g, p, (\xi_k),(\eta^k))$ provides a metrical framed $f(3, -1)$- structure if and only if (\ref{3.25}) and
\begin{equation}\label{3.28}
\gamma=\alpha,\ \ \ \lambda=a\beta,\ \  \ \
\rho=\kappa(a+b\tau),
\end{equation}
hold good.
\end{thm}
\begin{proof}
Using Lemma \ref{Dr}, it is easy to see that the metricity  condition
\[
{}^{CG}g(pX, pY)={}^{CG}g(X,
Y)-\eta^1(X)\eta^1(Y)-\eta^2(X)\eta^2(Y)-\eta^3(X)\eta^3(Y),
\]
of the framed $f(3, -1)$ structure characterized by (\ref{3.25}) is satisfied if and only if (\ref{3.28}) hold good. Thus the proof is complete.
\end{proof}
\section{On $(1,1)$-tensor sphere bundle}
Let $r$ be a positive number. Then the $(1,1)$ tensor sphere
bundle of radius $r$ over a Riemannian $(M,g)$ is the hypersurface
$T_{1r}^1(M)=\{(x,t)\in T^1_1M|G_x(t,t)=r^2\}$. It is easy to
check that the tensor field
\[
N=t^i_je_{\overline{j}},
\]
is a tensor field on $TM^1_1$ which is normal to $T_{1r}^1M$.

In general for any tensor field $A\in\Im^1_1(M)$, the vertical
lift $A^V$ is not tangent to $T_{1r}^1M$ at point $(x,t)$.
We define the tangential lift $A^T$ of a tensor field $A$ to $(x,t)\in T_{1r}^1M$ by   \\
\begin{equation}\label{4.29}
A^T_{(x,t)}=A^V_{(x,t)}-\frac{1}{r^2}G_{x}(A,t)N^V_{(x,t)}.
\end{equation}
Now, the tangent space $TT_{1r}^1M$ is spanned by $e_j$ and $
e_{\bar j}^T=\partial_{\overline{j}}-\frac{1}{r^2}\overline{t}^j_it^v_r\partial_{\overline{r}}
$. We notice there is the relation
$t^i_je_{\bar j}^T=0$, hence in any point of $T_{1r}^1M$ the vectors
$e_{\bar{j}}^T$; $\bar{j}=n+1,\ldots,n+n^2$,   span an $(n^2-1)$-
dimensional subspace of $TT_{1r}^1(M)$. Using (\ref{4.29}) and the
computation starting with the formula (\ref{3.7}), we see that the
Riemannian metric $\widetilde{g}$ on $T_{1}^1M$, induced from
$^{CG}g$, is completely determined by the identities
\begin{eqnarray}
\widetilde{g}(^TA,^TB)\!\!\!\!&=&\!\!\!\!a{}^V(G(A,B)-\frac{1}{r^2} G(t,A)G(t,B)),\nonumber\\
\widetilde{g}(^TA,^HY)\!\!\!\!&=&\!\!\!\!0,\label{g1}\\
\widetilde{g}(^HX,^HY)\!\!\!\!&=&\!\!\!\!{}^V(g(X,Y)),\nonumber
\end{eqnarray}
for all $X, Y\in\Im^1_0(M)$ and $A, B\in\Im^1_1(M)$, where $a$ is
constant that
satisfy $a>0$.\\
The bracket operation of tangential and horizontal vector fields
is given by the formulas
\[[e_{\bar{l}}^T,e_{\bar{j}}^T]=\frac{1}{r^2}(\overline{t}^l_t\delta^v_i\delta^j_r-\overline{t}^j_i\delta^v_t\delta^l_r)e_{\bar{r}}^T,
\]
\[[e_l,e_{\bar{j}}^T]=(\Gamma^v_{li}\delta^j_r-\Gamma^j_{lr}\delta^v_i)e_{\bar{r}}^T,\hskip
.6cm
\]
\[[e_l,e_j]=(R_{ljr}^{\ \ \ s}t^v_s-R_{ljs}^{\ \ \ v
}t^s_r)e^T_{\bar{r}}.
\]
\begin{prop}\label{2}
The Levi-Civita connection $\widetilde{\nabla}$, associated the
Riemannian metric $\widetilde{g}$ on the tensor bundle
$T_{1r}^1M$ has the form
\begin{eqnarray*}
\widetilde{\nabla}_{e_l}^{e_j}\!\!\!\!&=&\!\!\!\!\Gamma^r_{lj}e_r+\frac{1}{2}(R_{ljr}^{\
\ \ s}t^v_s-R_{ljs}^{\ \ \ v}t^s_r)e_{\bar{r}}^T,\\
\widetilde{\nabla}_{e_{\bar{l}}^T}^{e_j}\!\!\!\!&=&\!\!\!\!\frac{a}{2}(g_{ta}R_{\
\ j }^{sl\ r}t^a_s-g^{lb}R_{tsj}^{\ \ \ r}t^s_b)e_r
,\\
\widetilde{\nabla}_{e_l}^{e_{\bar{j}}^T}\!\!\!\!&=&\!\!\!\!\frac{a}{2}(g_{ia}R_{\
\ l }^{sj\ r}t^a_s-g^{jb}R_{isl}^{\ \ \ r}t^s_b)e_r
+(\Gamma^v_{li}\delta^j_r-\Gamma^j_{lr}\delta^v_i)e_{\bar{r}}^T,\\
\widetilde{\nabla}_{e_{\bar{l}}^T}^{e_{\bar{j}}^T}
\!\!\!\!&=&\!\!\!\!-\frac{1}{r^2}\overline{t}^j_i\delta^l_r\delta^v_te_{\bar{r}}^T.
\end{eqnarray*}
\end{prop}
\subsection{An almost paracontact structure on $T^1_{1r}M$}
In this section, we show that the
framed $f(3,-1)$- structure on $T^1_1M$, given by Theorem \ref{t1}, induces an almost paracontact
structure on $T^1_{1r}M$.

First, we show that $\xi_2$ and $\xi_3$ are unit normal vector fields with respect to the
metric $^{CG}g$. Let
\begin{equation}
x^i=x^i(u^\alpha),\ \ \ t^i_j=t^i_j(u^\alpha),\ \ \ \alpha \in
\{1,...,n\},
\end{equation}
be the local equations of $T^1_{1r}M$ in $T^1_1M$. Since
$\tau=t^i_jt^t_lg^{jl}g_{it}=r^2$, we have
\begin{equation}
\frac{\partial \tau}{\partial x^j}\frac{\partial x^j}{\partial
u^\alpha}+\frac{\partial \tau}{\partial t^k_h}\frac{\partial
t^k_h}{\partial u^\alpha}=0.\label{f}
\end{equation}
But we have
\begin{equation}
\frac{\partial \tau}{\partial
x^j}=2(\Gamma^k_{js}t^s_h-\Gamma^s_{jh}t^k_s)\bar{t}^h_k,\ \ \ \
\ \  \frac{\partial \tau}{\partial t^k_h}=2\bar{t}^h_k.\label{f1}
\end{equation}
By replacing (\ref{f1}) into (\ref{f}), we get
\begin{equation}
((\Gamma^k_{js}t^s_h-\Gamma^s_{jh}t^k_s)\frac{\partial
x^j}{\partial u^\alpha}+\frac{\partial t^k_h}{\partial
u^\alpha})\bar{t}^h_k=0.\label{28}
\end{equation}
The natural frame field on $T^1_{1r}M$ is represented by
\begin{equation}
\frac{\partial}{\partial u^\alpha}=\frac{\partial x^j}{\partial
u^\alpha}\frac{\partial}{\partial x^j}+\frac{\partial
t^k_h}{\partial u^\alpha}\frac{\partial}{\partial
t^k_h}=\frac{\partial x^j}{\partial
u^\alpha}e_j+((\Gamma^k_{js}t^s_h-\Gamma^s_{jh}t^k_s)\frac{\partial x^j}{\partial u^\alpha}+\frac{\partial
t^k_h}{\partial u^\alpha})e_{\bar{h}}.
\end{equation}
Then by (\ref{28}), we deduce that
\begin{equation}
^{CG}g(\frac{\partial}{\partial u^\alpha},\xi_3)
=\kappa(a+b\tau)((\Gamma^k_{js}t^s_h-\Gamma^s_{jh}t^k_s)\frac{\partial x^j}{\partial u^\alpha}+\frac{\partial
t^k_h}{\partial u^\alpha})\bar{t}^h_k=0.
\end{equation}
Similarly we obtain $^{CG}g(\frac{\partial}{\partial u^\alpha}, \xi_2)=0$. Thus $\xi_2$ and $\xi_3$ are orthogonal to any vector tangent to $T^1_{1r}M$. The vector field $\xi_1$
 is tangent to $T^1_{1r}M$ since $^{CG}g(\xi_1,\xi_2)=0$.
\begin{lem}\label{t5}
On $T^1_{1r}M$, we have
\[
\eta^2=\eta^3=0,\ \ \ p(X)=P(X)-\eta^1(X)\xi_1, \ \ \forall X\in\chi(T^1_{1r}M).
\]
\end{lem}
\begin{proof}
Using $\eta^i|_{T^1_{1r}M}(X)={}^{CG}g(X,\xi_i)=0, i=2,3$, the proof is obvious.
\end{proof}
We put $\xi_1|_{T^1_{1r}M}=\xi$, $\eta^1|_{T^1_{1r}M}=\eta$ and $p|_{T^1_{1r}M}=p$. Then Theorem \ref{t1} and Lemma \ref{t5} implie the following.
\begin{thm}\label{t6}
If (\ref{3.25}) holds, then the triple $(p, \xi, \eta)$ defines an almost paracontact structure on $T^1_{1r}M$, that is,

(i)\ \ $\eta(\xi)=1, \ \ p(\xi)=0,\ \ \eta\circ p=0$.

(ii)\ \ ${p}^2(X)=X-\eta(X)\xi,\ \ X\in\chi(T^1_{1r}M)$.\\
\end{thm}

It is easy to show that if (\ref{3.25}) and (\ref{3.28}) hold, then the Riemannian metric $\widetilde{g}$ satisfies
\begin{equation}
\widetilde{g}(pX, pY)=\widetilde{g}(X,Y)-\eta(X)\eta(Y),\ \ X,Y\in\chi(T^1_{1r}M).\label{29}
\end{equation}
By the equation  (\ref{29}) and Theorem \ref{t6},  we conclude the following.
\begin{thm}
If (\ref{3.25}) and (\ref{3.28}) hold then the ensemble $(p, \xi, \eta, \widetilde{g})$ defines an almost
metrical paracontact structure on the tangent sphere bundle $T^1_{1r}M$.
\end{thm}
\subsection{ Non-existence $(1,1)$- tensor sphere bundles space form}
The curvature tensor field $\widetilde{R}$ of the connection
$\widetilde{\nabla}$ is defined by the well-known formula
\[
\widetilde{R}(\widetilde{X},\widetilde{Y})\widetilde{Z}=\widetilde{\nabla}_{\widetilde{X}}\widetilde{\nabla}_{\widetilde{Y}}\widetilde{Z}-\widetilde{\nabla}_{\widetilde{Y}}\widetilde{\nabla}_{\widetilde{X}}\widetilde{Z}
-\widetilde{\nabla}_{[\widetilde{X},\widetilde{Y}]}\widetilde{Z},
\]
where $\widetilde{X}, \widetilde{Y},
\widetilde{Z}\in\Im^1_0(T^1_{1r}M)$. Using the above equation, Proposition \ref{2} and the local frame
$\{e_j,e_ {\bar{j}}^T\}_{\bar j=1}^{n+n^2}$ we obtain
\begin{eqnarray}
\widetilde{R}(e_m,e_l)e_j\!\!\!\!&=&\!\!\!\!HHHH^r_{mlj}e_r+HHHT^{\bar{r}}_{mlj}e_{\bar{r}}^T,\label{rie2}\\
\widetilde{R}(e_m,e_l)e_{\bar
{j}}^T\!\!\!\!&=&\!\!\!\!HHTH^r_{ml{\bar{j}}}e_r
+HHTT^{\bar{r}}_{ml\bar{j}}e_{\bar{r}}^T,\\
\widetilde{R}(e_m,e_{\bar{l}}^T)e_j\!\!\!\!&=&\!\!\!\!HTHH^r_{m\bar{l}j}e_r+HTHT^{\bar{r}}_{m\bar{l}j}e_{\bar{r}}^T,\\
\widetilde{R}(e_{m},e_{\bar{l}}^T)e_{\bar{j}}^T\!\!\!\!&=&\!\!\!\!HTTH^r_{m\overline{l}\bar{j}}e_r,\label{rie1}\\
\widetilde{R}(e_{\bar{m}}^T,e_{\bar{l}}^T)e_{j}\!\!\!\!&=&\!\!\!\!TTHH^r_{{\bar{m}}\overline{l}j}e_r,
\\
\widetilde{R}(e_{\bar{m}}^T,e_{\bar{l}}^T)e_{\bar{j}}^T
\!\!\!\!&=&\!\!\!\!TTTT^{\bar{r}}_{\bar{m}\bar{l}\bar{j}}e_{\bar{r}}^T,\label{n2}
\end{eqnarray}
where
\begin{eqnarray*}
HHHH^{\ \ \ \ r}_{mlj }\!\!\!\!&=&\!\!\!\!R^{\ \ \ \
r}_{mlj}+\frac{a}{4}\{g_{ka}(R^{sh\ r}_{\ \ m}R^{\ \ \
p}_{ljh}-R^{sh\ r}_{\ \ l}R^{\ \ \ p}_{mjh}-2R^{sh\ r}_{\ \
j}R_{mlh}^{\ \ \ p})t^a_st^k_p
\nonumber\\
\!\!\!\!&&\!\!\!\!+g_{ka}(R^{sh\ r}_{\ \ l}R^{\ \ \
k}_{mjp}-R^{sh\ r}_{\ \ m}R^{\ \ \ \ k}_{ljp}+2R^{sh\ r}_{\ \
j}R_{mlp}^{\ \ \ k})t^a_st^p_h
\nonumber\\
\!\!\!\!&&\!\!\!\!+g^{hb}(R^{\ \ \ r}_{kpl}R^{\ \ \ \
s}_{mjh}-R^{\ \ \ \ r}_{kpm}R^{\ \ \ s}_{ljh}+2R^{\ \ \ \ r}_{kpj
}R_{mlh}^{\ \ \ \ s})t^p_bt^k_s
\nonumber\\
\!\!\!\!&&\!\!\!\!+g^{hb}(R^{\ \ \ r}_{ksm}R^{\ \ \ \
k}_{ljp}-R^{\ \ \ \ r}_{ksl}R^{\ \ \ k}_{mjp}-2R^{\ \ \ \ r}_{ksj
}R_{mlp}^{\ \ \ \ k})t^s_bt^p_h\},
\end{eqnarray*}
\[
HHHT^{\ \ \ \ \overline{r}}_{mlj}=\frac{1}{2}\{\nabla_mR_{ljr}^{\
\ \ s}t^v_s-\nabla_lR_{mjr}^{\ \ \ \ s}t^v_s+\nabla_lR_{mjs}^{\ \
\ \ v}t^s_r-\nabla_mR_{ljs}^{\ \ \ v}t^s_r\},
\]
\[
HHTH^{\ \ \ \ r}_{ml\overline{j}}=\frac{a}{2}\{g_{ia}\nabla_mR_{\
\ l}^{sj\ r}t^a_s-\nabla_lR_{\ \ m}^{sj\ r}
t^a_s+g^{jb}\nabla_lR_{ism}^{\ \ \ r }t^s_b-\nabla_mR_{isl}^{\ \
\ r}t^s_b\},
\]
\begin{eqnarray*}
HHTT^{\ \ \ \ \overline{r}}_{ml\overline{j}
}\!\!\!\!&=&\!\!\!\!R^{\ \ \ \ v}_{mli}\delta^j_r-R_{mlr}^{\ \ \ j
}\delta^v_i+\frac{a}{4}\{g_{ia}(R^{\ \ \ \ s}_{mhr}R^{pj\ h }_{\ \
l }-R^{\ \ \ s}_{lhr}R^{pj\ h}_{\ \ m})t^v_st^a_p\nonumber\\
\!\!\!\!&&\!\!\!\!+g_{ia}(R^{\ \ \ \ v}_{lhp}R^{sj\ h}_{\ \
m}-R^{\ \ \ \ v}_{mhp}R^{sj\ h}_{\ \ l})t^a_st^p_r +g^{jb}(R^{\ \
\ \ s}_{lhr}R^{\ \ \ \ h}_{ipm}
\nonumber\\
\!\!\!\!&&\!\!\!\!-R^{\ \ \ \ s}_{mhr}R^{\ \ \
h}_{ipl})t^p_bt^v_s+g^{jb}(R^{\ \ \ \ v}_{mhs}R^{\ \ \
h}_{ipl}-R^{\ \ \ \ v}_{lhs}R^{\ \ \ \ h}_{ipm})t^s_rt^p_b\}\nonumber\\
\!\!\!\!&&\!\!\!\!+\frac{1}{r^2}(R_{mlr}^{\ \ \ \ s
}t^v_s-R_{mls}^{\ \ \ \ v }t^s_r)\overline{t}^j_i,
\end{eqnarray*}
\[
HTHH^{\ \ \ r}_{m\overline{l}j}=\frac{a}{2}\{g_{ta}\nabla_mR_{\ \
j}^{sl\ r}t^a_s-g^{lb}\nabla_mR_{tsj}^{\ \ \ r}t^s_b\},
\]
\begin{eqnarray*}
HTHT^{\ \ \ \ \overline{r}}_{m\overline{l}j
}\!\!\!\!&=&\!\!\!\!-\frac{1}{2}(R^{\ \ \ \ l}_{mjr
}\delta^v_t-R_{mjt}^{\ \ \ \ v}\delta^l_r)
+\frac{a}{4}\{g_{ta}R^{pl\ \ h}_{\ \ j}R_{mhr}^{\ \ \ \ \
s}t^v_st^a_p\nonumber\\
\!\!\!\!&&\!\!\!\!-g^{lb}R^{\ \ \ h}_{tpj }R^{\ \ \ \ \
s}_{mhr}t^v_st^p_b -g_{ta}R^{sl\ h}_{\ \ j}R^{\ \ \ \ \
v}_{mhp}t^p_rt^a_s\nonumber\\
\!\!\!\!&&\!\!\!\!+g^{lb}R^{\ \ \ h}_{tpj}R^{\ \ \ \ v
}_{mhs}t^s_rt^p_b\},
\end{eqnarray*}
\begin{eqnarray*}
HTTH^{\ \ \ \ r}_{m\overline{l}\overline{j}
}\!\!\!\!&=&\!\!\!\!\frac{a}{2}(g^{jl}R_{itm}^{\ \ \
r}-g_{it}R^{lj\ \ r}_{\ \ m})+\frac{a^2}{4}\{g_{ta}R^{sl\ r}_{\ \
h}g^{jb}R^{\ \ \ \ h}_{ipm}t^a_st^p_b
\nonumber\\
\!\!\!\!&&\!\!\!\!-g_{ta}R^{sl\ r}_{\ \ h}g_{ib}R^{pj \ h}_{\ \
m}t^a_st^b_p+g^{lb}R^{\ \ \ r}_{tph}g_{ia}R^{sj\ h}_{\ \
m}t^p_bt^a_s
\nonumber\\
\!\!\!\!&&\!\!\!\!-g^{la}R^{\ \ \ \ r}_{tsh}g^{jb}R^{\ \ \
h}_{ipm}t^s_at^p_b\}-\frac{a}{2r^2}(g_{ta}R^{sl \ r}_{\ \
m}t^a_s\nonumber\\
\!\!\!\!&&\!\!\!\!-g^{lb}R^{\ \ \ \
r}_{tsm}t^s_b)\overline{t}^j_i,
\end{eqnarray*}
\begin{eqnarray*}
TTHH^{\ \ \ \ r}_{\overline{m}\overline{l}j
}\!\!\!\!&=&\!\!\!\!a(g_{tn}R^{ml\ r}_{\ \ j}-g^{lm}R_{tnj}^{\ \ \
\ r})+\frac{a^2}{4}\{g_{na}R^{sm\ r}_{\ \ \ h}g_{tb}R^{pl\ h}_{\ \
j}t^a_st^b_p
\nonumber\\
\!\!\!\!&&\!\!\!\!-g_{ta}R^{sl\ r}_{\ \ h }g_{nb}R^{pm\ h }_{\ \
j}t^a_st^b_p+g_{ta}R^{sl\ r}_{\ \ h}g^{mb}R^{\ \ \ \
h}_{npj}t^a_st^p_b
\nonumber\\
\!\!\!\!&&\!\!\!\!-g_{na}R^{sm\ r}_{\ \ h}g^{lb}R_{tpj}^{\ \ \ h
}t^a_st^p_b+g^{lb}R^{\ \ \ r}_{tph}g_{na}R^{sm\ h }_{\ \ j
}t^p_bt^a_s
\nonumber\\
\!\!\!\!&&\!\!\!\!-g^{mb}R_{nph}^{\ \ \ r }g_{ta}R^{sl\ h }_{\ \ j
}t^p_bt^a_s+g^{ma}R^{\ \ \ \ r}_{nsh}g^{lb}R^{\ \ \ h
}_{tsj }t^p_bt^s_a\nonumber\\
\!\!\!\!&&\!\!\!\!-g^{la}R_{tsh}^{\ \ \ r }g^{mb}R^{\ \ \ \ h
}_{npj }t^p_bt^s_a\},
\end{eqnarray*}
\begin{eqnarray*}
TTTT^{\ \ \ \
\overline{r}}_{\overline{m}\overline{l}\overline{j}}\!\!\!\!&=&\!\!\!\!
\frac{1}{r^4}(\overline{t}^m_n\overline{t}^j_i\delta_r^l\delta^v_t
-\overline{t}^l_t\overline{t}^j_i\delta_r^m\delta^v_n)+\frac{1}{r^2}
(g^{lj}g_{ti}\delta^m_r\delta^v_n\nonumber\\
\!\!\!\!&&\!\!\!\! -g^{mj}g_{ni}\delta^m_r\delta^v_n ).
\end{eqnarray*}
\begin{thm}\label{n5}
$(1, 1)$-tensor sphere bundle $T_{1r}^1M$, with the Riemannian metric
$\widetilde{g}$ induced from the metric ${}^{CG}g$ on $T_{1}^1M$, has never constant
sectional curvature.
\end{thm}
\begin{proof}
It is known that the curvature tensor field of the
Riemannain manifold $(T_{1r}^1M,\widetilde{g})$ with constant
section
curvature $k$, satisfy the relation\\
\begin{equation}
\widetilde{R}(\widetilde{X},\widetilde{Y})\widetilde{Z}=k\{\widetilde{g}(\widetilde{Y},\widetilde{Z})
\widetilde{X}-\widetilde{g}(\widetilde{X},\widetilde{Z})\widetilde{Y}\},\label{r1}
\end{equation}
where $\widetilde{X}, \widetilde{Y},
\widetilde{Z}\in\Im^1_0(T^1_{1r}M)$. Let
$(T_{1r}^1M,\widetilde{g})$  have constant sectional curvature
$k$. Then we have
\begin{equation}
\widetilde{R}(e_{\bar{m}}^T,e_{\bar{l}}^T)e_{\bar{j}}^T-k\{\widetilde{g}(e^T_{\bar{l}},e^T_{\bar{j}})e_{\bar{m}}^T
-\widetilde{g}(e^T_{\bar{m}},e^T_{\bar{j}})e_{\bar{l}}^T\}=0.\label{lei1}
\end{equation}
Using (\ref{lei1}) and (\ref{n2}), we get
\begin{equation}
\frac{1-k
r^2a}{r^2}[g_{ti}g^{lj}\delta^m_r\delta^v_n-g_{ni}g^{mj}\delta^l_r\delta^v_t+
\frac{1}{r^2}(\overline{t}^m_n\overline{t}^j_i\delta^l_r\delta^v_t-\overline{t}^l_t\overline{t}^j_i\delta^m_r\delta^v_n)]=0\label{n3}.
\end{equation}
Using the above equation and Lemma \ref{esileila} we deduce $k\neq 0$ and $a=\frac{1}{k r^2}$.
Since $(T_{1r}^1M,\widetilde{g})$  has constant sectional
curvature $k$, we have
\begin{equation}
\widetilde{R}(e_m,e_l)e_j-k\{\widetilde{g}(e_l,e_j)e_m-\widetilde{g}(e_m,e_j)e_l\}=0\label{lei2}
%
\end{equation}
(\ref{rie2}) and (\ref{lei2}) give us
 \begin{eqnarray}
\!\!\!\!&&\!\!\!\!R^{\ \ \ \
r}_{mlj}-k(g_{lj}\delta^r_m-g_{mj}\delta^r_l)+\frac{a}{4}\{g_{ka}(R^{sh\
r}_{\ \ m}R^{\ \ \
p}_{ljh}\nonumber\\
 \!\!\!\!&&\!\!\!\!-R^{sh\ r}_{\ \ l}R^{\ \ \ p}_{mjh}-2R^{sh\ r}_{\ \
j}R_{mlh}^{\ \ \ p})t^a_st^k_p +g_{ka}(R^{sh\ r}_{\ \ l}R^{\ \ \
k}_{mjp}\nonumber\\
\!\!\!\!&&\!\!\!\!-R^{sh\ r}_{\ \ m}R^{\ \ \ \ k}_{ljp}+2R^{sh\
r}_{\ \ j}R_{mlp}^{\ \ \ k})t^a_st^p_h+g^{hb}(R^{\ \ \
r}_{ksm}R^{\ \ \ \
k}_{ljp}\nonumber\\
\!\!\!\!&&\!\!\!\!-R^{\ \ \ \ r}_{ksl}R^{\ \ \ k}_{mjp}-2R^{\ \ \
\ r}_{ksj }R_{mlp}^{\ \ \ \ k})t^s_bt^p_h +g^{hb}(R^{\ \ \
r}_{kpl}R^{\ \ \ \
s}_{mjh}\nonumber\\
\!\!\!\!&&\!\!\!\!-R^{\ \ \ \ r}_{kpm}R^{\ \ \ s}_{ljh}+2R^{\ \ \
\ r}_{kpj }R_{mlh}^{\ \ \ \ s})t^p_bt^k_s \}=0,\label{lei4}
\end{eqnarray}

Differentiating the expression (\ref{lei4}) two times, in the
tangential coordinates $x^{\bar{j}}; \bar{j}=1,\ldots, n+n^2$, we
conclude
\begin{equation}\label{4.49}
R^{\ \ \ \
r}_{mlj}=k(g_{lj}\delta^r_m-g_{mj}\delta^r_l).
\end{equation}
Also, we have
\begin{equation}
\widetilde{R}(e_{\bar{m}}^T,e_l)e_{\bar{j}}^T-k\{\widetilde{g}(e_l,e_{\bar{j}}^T)e_{\bar{m}}^T
-\widetilde{g}(e_{\bar{m}}^T,e_{\bar{j}}^T)e_l\}=0.\label{n4}
\end{equation}
Setting $a=\frac{1}{kr^2}$, replacing (\ref{4.49}) in (\ref{rie1}) and then using (\ref{n4}) we obtain
\begin{eqnarray*}
\!\!\!\!&&\!\!\!\!-\frac{1}{2r^2}[g^{jl}(g_{tm}\delta^r_i-g_{im}\delta^r_t+2g_{it}\delta^r_m)+g_{it}
(g^{jr}\delta^l_m-g^{lr}\delta^j_m)]\nonumber\\
\!\!\!\!&&\!\!\!\!-\frac{1}{4r^4}[g_{ta}g^{jb}(g_{pm}g^{sr}\delta^l_it^a_st^p_b-g_{im}g^{sr}t^a_st^l_b-g_{pm}g^{lr}
t^a_it^p_b+g_{im}g^{lr}t^a_pt^p_b)
\nonumber\\
\!\!\!\!&&\!\!\!\!+g_{ta}g_{ib}(g^{sr}g^{jl}t^a_st^b_m-g^{sr}g^{lp}\delta^j_mt^a_st^b_p+g^{lr}g^{sp}\delta^j_mt^a_st^b_p
-g^{lr}g^{js}t^a_st^b_m)
\nonumber\\
\!\!\!\!&&\!\!\!\!+g^{la}g^{jb}(g_{sp}g_{im}\delta^r_tt^s_at^p_b-g_{si}g_{pm}\delta^r_tt^s_at^p_b+g_{ti}g_{pm}t^r_at^p_b
-g_{tp}g_{im}t^r_at^p_b)\nonumber\\
\!\!\!\!&&\!\!\!\!+g_{ia}g^{lb}(\delta^j_m\delta^r_t
t^p_bt^a_p-\delta^r_tt^j_bt^a_m-\delta^j_mt^r_bt^a_t+\delta^j_t
t^r_bt^a_m)]
\nonumber\\
\!\!\!\!&&\!\!\!\!+\frac{1}{2r^4}[(g_{ta}g^{sr}\delta^l_mt^a_s-g_{ta}g^{lr}t^a_m-g_{sm}g^{lb}\delta^r_tt^s_b+g_{tm}g^{lb}t^r_b
+2\delta^r_m\bar{t}^l_t)\bar{t}^j_i]=0.
\end{eqnarray*}
From the above equation in the point $(x^i, t^j_i)=(x^i, \delta^j_i)\in T^1_1M$ we get
\[
-\frac{1}{2r^2}[g^{jl}(g_{tm}\delta^r_i-g_{im}\delta^r_t+2g_{it}\delta^r_m)+g_{it}
(g^{jr}\delta^l_m-g^{lr}\delta^j_m)]+\frac{1}{r^4}\delta^r_m\delta^l_t\delta^j_i=0,
\]
which is a contradiction. Thus we conclude that the manifold
$(T_{1r}^1M,\widetilde{g})$ may never be a space form.
\end{proof}
Since for Sasaki metric $^S{g}$ we have $a=1$, then by using
Theorem \ref{n5} we have
\begin{cor}
The $(1, 1)$-tensor sphere bundle $T_{1r}^1M$, endowed with the metric
induced by the Sasaki metric $^S{g}$ from $T_{1}^1M$, is never a space form.
\end{cor}

\end{document}